\documentclass[12pt]{amsproc}

\usepackage{amsmath, amssymb}

\newtheorem{theorem}{\sc Theorem}[section]

\newtheorem{lemma}[theorem]{\sc Lemma}
\newtheorem{proposition}[theorem]{\sc Proposition}
\newtheorem{corollary}[theorem]{\sc Corollary}

\begin{document}

\title[]
{BFC-theorems for higher commutator subgroups}
\author[E. Detomi]{Eloisa Detomi}
\address{Dipartimento di Matematica ``Tullio Levi-Civita'', Universit\`a di Padova \\
 Via Trieste 63\\ 35121 Padova \\ Italy}
\email{detomi@math.unipd.it}
\author[M. Morigi]{Marta Morigi}
\address{Dipartimento di Matematica, Universit\`a di Bologna\\
Piazza di Porta San Donato 5 \\ 40126 Bologna\\ Italy}
\email{marta.morigi@unibo.it}
\author[P. Shumyatsky]{Pavel Shumyatsky}
\address{Department of Mathematics, University of Brasilia\\
Brasilia-DF \\ 70910-900 Brazil}
\email{pavel@unb.br}

\thanks{This research was partially supported by MIUR (Prin 2015: ``Group theory and applications"). 
The third author was also supported by FAPDF and CNPq.}

\subjclass[2010]{20E45; 20F12; 20F24.}
\keywords{Conjucagy classes, verbal subgroups, commutators.}

\begin{abstract}
 A BFC-group is a group in which all conjugacy classes are finite with bounded size. In 1954 B. H. Neumann discovered that if $G$ is a BFC-group
 then the derived group $G'$ is finite. Let $w=w(x_1,\dots,x_n)$ be a multilinear commutator. We study groups in which the conjugacy classes containing
 $w$-values are finite of bounded order. Let $G$ be a group and let $w(G)$ be the verbal subgroup of $G$ generated by all $w$-values. We prove that if $|x^G|\le m$ for every $w$-value $x$, then the derived subgroup of $w(G)$ is finite of order bounded by a 
function of $m$ and $n$.  If $|x^{w(G)}|\le m$ for every $w$-value $x$, then  $[w(w(G)),w(G)]$ is finite of order bounded by a 
function of $m$ and $n$. 
\end{abstract}

\maketitle
\section{Introduction} Given a group $G$ and an element $x\in G$, we write $x^G$ for the conjugacy class containing $x$. Of course, 
if the number of elements in $x^G$ is finite, we have $|x^G|=[G:C_G(x)]$. A group is said to be a BFC-group if its conjugacy classes are 
finite and of bounded size. One of the most famous of B. H. Neumann's theorems says that in a BFC-group the commutator subgroup $G'$ is 
finite \cite{bhn}. It follows that if $|x^G|\leq m$ for each $x\in G$, then the order of $G'$ is bounded by a number depending only on $m$. 
A first explicit bound for the order of $G'$ was found by J. Wiegold \cite{wie}, and the best known was obtained in 
\cite{gumaroti} (see also \cite{neuvoe} and \cite{sesha}).

The recent article \cite{dieshu} deals with groups $G$ in which conjugacy classes containing commutators are bounded. By a commutator we mean any element $x\in G$ which can be written in the form $$x=[x_1,x_2]=x_1^{-1}x_2^{-1}x_1x_2$$ for suitable $x_1,x_2\in G$. The results obtained in \cite{dieshu} can be summarized as follows.

\begin{theorem}\label{gla}
Let $m$ be a positive integer and $G$ a group. If $|x^G|\leq m$ for any commutator $x$, then $|G''|$ is finite and $m$-bounded. If $|x^{G'}|\leq m$ for any commutator $x$, then $|\gamma_3(G')|$ is finite and $n$-bounded.
\end{theorem}

Here $G''$ denotes the second commutator subgroup of $G$ and $\gamma_3(G')$ denotes the third term of the lower central series of $G'$. Throughout the article we use the expression ``$(a,b,\dots)$-bounded" to mean that a quantity is finite and bounded by a certain number depending only on the parameters $a,b,\dots$. 

Comparing Neumann's theorem with Theorem \ref{gla} one cannot help but wondering whether these results are parts of some more general phenomenon. The purpose of the present article is to address this question. 

Given a group-word $w=w(x_1,\dots,x_k)$, the verbal subgroup $w(G)$ of a group $G$ determined by $w$ is the subgroup generated by the set $G_w$ consisting of all values $w(g_1,\ldots,g_k)$, where $g_1,\ldots,g_k$ are elements of $G$. In particular, we will work with multilinear commutator words. These are words which are obtained by nesting commutators, but using always different variables. More formally, the word $w(x)=x$ in one variable is a multilinear commutator; if $u$ and $v$ are  multilinear commutators involving different variables then the word $w=[u,v]$ is a multilinear commutator, and all multilinear commutators are obtained in this way. Note that in the literature the multilinear commutators are sometimes called outer commutator words. In the recent article by A. Shalev \cite{shalev} they are called general commutator words. Examples of multilinear commutators include the familiar lower central words $\gamma_k(x_1,\dots,x_k)=[x_1,\dots,x_k]$ and derived words $\delta_k$,  
on $2^k$ variables, defined recursively by
$$\delta_0=x_1,\qquad \delta_k=[\delta_{k-1}(x_1,\ldots,x_{2^{k-1}}),\delta_{k-1}(x_{2^{k-1}+1},\ldots,x_{2^k})].$$
 Of course, $\gamma_k(G)$ is the $k$-th  term of the lower central series of $G$ while $\delta_k(G)=G^{(k)}$ is the $k$-th  term of the derived series.

We establish the following results.

\begin{theorem}\label{aaa} Let $w=w(x_1,\dots,x_n)$ be a multilinear commutator, 
 and let $G$ be a group such that $|x^G|\leq m$ 
for every $w$-value $x$ in $G$. Then the commutator subgroup of $w(G)$ has finite $(m,n)$-bounded order.\end{theorem}

\begin{theorem}\label{bbb} Let $w=w(x_1,\dots,x_n)$ be a multilinear 
commutator,
 and let $G$ be a group such that $|x^{w(G)}|\leq m$ 
for every  $w$-value $x$ in $G$. Then $[w(w(G)),w(G)]$ has finite $(m,n)$-bounded order.\end{theorem}

It is easy to see that the theorem of Neumann can be obtained from each of the above theorems by considering the case where $w(x)=x$ while Theorem \ref{gla} follows by taking $w(x_1,x_2)=[x_1,x_2]$.

A number of further results of similar nature can be deduced from from rather general Theorem \ref{together} which states that for any subgroup $K$ of $G$ such that $w(G)\leq K$ and $|a^K|\leq m$ for each $a\in G_w$, the order of $[w(K),w(G)]$ is $(m,w)$-bounded (see Section 4 for details).  For example, let $w=[[x_1,x_2],[x_3,x_4]]$ be the metabelian word. If $|x^G|\leq m$ for each $w$-value in $G$, then $G^{(3)}$ has finite $m$-bounded order. If $|x^{G'}|\leq m$ for each $w$-value in $G$, then $\gamma_3(G^{(2)})$ has finite $m$-bounded order.

The proofs of Theorem \ref{aaa} and Theorem \ref{bbb} are based on rather specific combinatorial techniques which were developed in \cite{fernandez-morigi, DMS1,DMS-revised} for handling multilinear commutator words. 
It seems unlikely that results of similar nature for arbitrary words hold.

\section{Preliminary results}

Throughout this section, $w$ will be a fixed word. 

Let $G$ be a group generated by a set $X$ such that $X = X^{-1}$. Given
an element $g\in G$, we write $l_X(g)$ for the minimal number $l$ with the
property that $g$ can be written as a product of $l$ elements of $X$. Clearly, $l_X(g)=0$ if and only if $g = 1$. We call $l_X(g)$ the length of $g$ with respect to $X$. The following result is Lemma 2.1 in \cite{dieshu}.

\begin{lemma}\label{2.1g}  Let $H$ be a group generated by a set $X = X^{-1}$ and let $K$ be a subgroup of finite index $m$ in $H$. Then each coset $Kb$ contains an element $g$ such that $l_X(g)\le m-1$.
\end{lemma}

In the sequel the above lemma will be used in the situation where $H=w(G)$ and $X=G_w\cup G_w^{-1}$ is the set of $w$-values and their inverses in $G$. Therefore we will write $l(g)$ to denote the smallest number such that the element $g\in w(G)$ can be written as a product of as many $w$-values or their inverses.

Recall that if $G$ is a group, $a\in G$ and $H$ is a subgroup of $G$, then $[H,a]$ denotes the subgroup of $G$ generated by all commutators of the form $[h,a]$, where $h\in H$. It is well-known that $[H, a]$ is normalized by $a$ and $H$.

The following result is analogous to Lemma 2.3 in \cite{dieshu}.  

\begin{lemma} \label{2.3} 
Let $w$ be a word and let $G$  be a group such that for every $x\in G_w$ the centralizer $C_{w(G)}(x)$ of $x$ in $w(G)$ has finite index at most $m$ in $w(G)$. Then  $[w(G),x]$ has $m$-bounded order for every $x\in G_w$.
\end{lemma}
\begin{proof} Let $W=w(G)$. Since $C_W (x)$ has index at most $m$ in $W$, by Lemma \ref{2.1g} we can choose elements $y_1,\dots,y_m$ such that $l(y_i)\le m-1$ and $[W,x]$ is generated by the commutators $[y_{i},x]$. For each
$i=1,\dots,m$ write $y_i=y_{i\,1}\cdots y_{i\,{m-1}}$, where $y_{i\,j}\in G_w\cup G_w^{-1}$. The standard commutator identities show that $[y_i,x]$ can be written as a product of conjugates in $W$ of the commutators $[y_{ij} ,x]$. Let $h_1,\dots,h_s$ be the conjugates in $W$ of elements from the set $\{x, y_{ij}; 1\le i, j \le m\}$. Since $C_W (h)$ has finite index at most $m$ in $W$ for each $h\in G_w$, it follows that $s$ is $m$-bounded. Let $ T = \langle h_1 , \dots, h_s\rangle$. It is clear that $[W, x]\le T'$ and
so it is sufficient to show that $T'$ has finite $m$-bounded order. 
 As $|W:C_{W} (h_i)| \le m$ for every $i$, 
 it follows that the center $Z(T )$ has index at most $m^s$ in $T$. Thus,
Schur's theorem \cite[10.1.4]{robinson} tells us that $T'$ has finite $m$-bounded order,
as required.
\end{proof}

\begin{lemma}\label{2.3b} 
 Let $w$ be a word, let $G$ be a 
 group and let $K$ be a subgroup of $G$ containing $w(G)$ such that $|K:C_K(x)|\le m$ for every $x\in G_w$. Then $[w(G),x]^K$ has $m$-bounded order for every  $x\in G_w$.
\end{lemma}

\begin{proof} Choose $x\in G_w$ and let $W=w(G)$. By Lemma \ref{2.3}, $[W,x]$ has $m$-bounded order. Observe that $[W,x]$ has at most $m$ conjugates in $K$ and the conjugates normalize each other. Thus, $[W,x]^K$ is a product of at most $m$ subgroups that normalize each other and have $m$-bounded order. The lemma follows.
\end{proof}

The following lemma can be seen as a development related to Lemma 2.4 in \cite{dieshu} and Lemma 4.5 in \cite{wie}. It plays a central role in our arguments.  

\begin{lemma} \label{basic-light}  Let $m,j$ be positive integers and $w=w(x_1,\dots ,x_n)$ a word. Let $G$ be a group and $K$ be a subgroup of $G$ 
containing $w(G)$ such that $|K:C_K(x)|\le m$ for every $x\in G_w$. Assume that $M$ is a normal subgroup of $K$ of index $j$. 
Choose $k_i\in K$ for $i=1,\dots,n$. Then there exist elements $\tilde k_i\in k_i M$, for $i=1,\dots,n$, and a normal subgroup $\tilde M$ of $M$ 
of finite $(j,m)$-bounded index, such that the order of $[w(G),w(\tilde k_1 \tilde M ,\dots,\tilde k_n \tilde M)]^K$ is finite and $m$-bounded.
\end{lemma}
\begin{proof}
Let $W=w(G)$. 
Consider the set 
$$S=\{w( k_1u_1,\dots,k_nu_n)|u_1,\dots,u_n\in M \}.$$ 
Choose in $S$ an element  $a=w(\tilde k_1,\dots,\tilde k_n)$ such that the number of conjugates of $a$ in $W$ is maximal among the elements of $S$, that is  
$|a^W|\ge |g^W|$ for all $g\in S$.
 
 By Lemma \ref{2.1g}  we can choose $b_1,\dots, b_r\in W$ such that $l(b_i) \le m-1$ and $a^W = \{a^{b_i} | i = 1, \dots, r\}$. Set $\tilde M  =M\cap ( C_K ( \langle b_1 ,\dots, b_r\rangle ))_K$ (i.e. $\tilde M$ is the intersection of $M$ and all $K$-conjugates of $ C_K (\langle b_1 ,\dots, b_r \rangle)$). Since $l(b_i) \le m-1$ and $C_K(x)$ has index at most $m$ in $K$ for each  $x \in G_w$, the subgroup  $C_K(\langle b_1, \dots , b_r\rangle)$ has $m$-bounded index in $K$, and so $\tilde M$ has $(j,m)$-bounded index. 

Consider the element  $w(\tilde k_1v_1,\dots,\tilde k_nv_n)\in S$ where each  $v_i\in \tilde M$, for $i=1,\dots,n$.  we have 
$$w(\tilde k_1v_1,\dots,\tilde k_nv_n)=va,$$
 for some $v\in \tilde M\le  C_K ( b_1 ,\dots, b_r )$. 
It follows that $(va)^{b_i} = va^{b_i}$ for each $i =1,\dots, r$. Therefore the elements $va^{b_i}$ form the conjugacy class $(va)^W$ because they are all different and their number is the allowed maximum. So, for an arbitrary element $h\in W$ there exists $b\in\{b_1 ,\dots, b_r\}$ such that
$(va)^h= va^b$ and hence $v^h a^h = va^b$. Therefore $[h, v] = v^{-h}v=a^h a^{-b}$ and so $[h, v]^a =a^{-1} a^h a^{-b} a = [a,h][b,a] \in [W,a].$
 Thus $[W,v]^a \le [W,a]$ and $$[W, va]=[W,a] [W,v]^a \le [W, a].$$ 
Therefore $[W,w(\tilde k_1 \tilde M ,\dots,\tilde k_n \tilde M)]\le [W,a]$.
Lemma \ref{2.3b} states that $[W,a]^K$ has $m$-bounded order. The result follows.
\end{proof}

\section{Combinatorics of multilinear commutators}

We will need some machinery concerning combinatorics of commutators, so we now recall some notation from the paper \cite{DMS-revised}.
 
 Throughout this section $w=w(x_1, \dots, x_n)$ is a multilinear commutator. 
If $A_1,\dots,A_n$ are subsets of a group $G$, we write
 $w(A_1, \dots , A_n)$ for the subgroup generated by the  $w$-values $w(a_1,\dots,a_n)$ with $a_i\in A_i$.  
  
Let $I$ be a subset of $\{1,\dots,n\}$. 
  Suppose that we have a family $A_{i_1}, \dots , A_{i_s}$ of subsets of $G$ with indices running  over $I$ and another family 
  $B_{l_1}, \dots , B_{l_t}$ of subsets with indices  running  over $\{1, \dots ,n \} \setminus I.$ 
 We write 
 $$w_I(A_i ; B_l)$$ 
 for $w(X_1, \dots , X_n)$, where $X_k=A_k$ if $k \in I$, and $X_k=B_k$  otherwise. 
 On the other hand, whenever $a_i\in A_i$ for $i\in I$ and $b_l\in B_l$ for $l\in \{1,\dots,n\}\setminus I$, the symbol 
 $w_I(a_i;b_l)$ stands for the element $w(x_1, \dots , x_n)$, where $x_k=a_k$ if $k \in I$, and $x_k=b_k$ otherwise.
  Sometimes, for the sake of shortness, we will omit the indices and we will simply write
  $$w_I(A ; B),$$
 for $w(X_1, \dots , X_n)$, where $X_k=A$ if $k \in I$, and $X_k=B$  otherwise. 
 
 The next lemma is Lemma 2.4 in \cite{DMS-revised}.

\begin{lemma}\label{2.1-conjugates} 
Let $w=w(x_1, \dots, x_n)$ be a multilinear commutator. 
Assume that  $H$ is a normal subgroup of a group $G$. 
Let $ g_1, \dots , g_n \in G$, $h \in H$  and fix $s \in \{1, \dots, n\}$. 
Then there exist 
 $y_j \in g_j^H$,  for  $j=1,\dots, n$, such that 
 \begin{eqnarray*}
 w_{\{s\}}(g_sh; g_l)=w(y_1, \dots,y_n) w_{\{s\}}(h; g_l). 
\end{eqnarray*}
\end{lemma}

The following corollaries are special
 cases of Lemma 2.5  and Lemma 4.1 of \cite{DMS-revised}, respectively. For the reader's convenience, we include the proofs. 

\begin{corollary}\label{uno2}
Let $w=w(x_1, \dots, x_n)$ be a multilinear commutator. Let $K$ be a group and $M$ a  normal subgroup of $K$. Assume that $$w( k_1M,\dots,k_nM)=1$$ for some elements $k_i\in K$. Let $I$ be a proper subset of $\{1,\dots,n\}$. 
Then 
$$w_I(k_iM;M )=1.$$
\end{corollary}
\begin{proof}
It is sufficent to prove that the result holds when $|I|=n-1$, and then repeatedly apply this case, with $k_i=1$ if needed. 
Let $\{s\}=  \{1,\dots,n\} \setminus I$. Clearly $w_I(k_iM;M )=w_{\{s\}}(M; k_lM),$
 so we have to prove that $w_{\{s\}}(M; k_lM)=1.$

Choose an element $w_{\{s\}}(m_s; k_lm_l) \in w_{\{s\}}(M; k_lM)$ and consider the element $w_{\{s\}}(k_sm_s; k_lm_l)$. 
By Lemma \ref{2.1-conjugates}, 
\begin{eqnarray}\label{11}
 w_{\{s\}}(k_sm_s; k_lm_l)=w(y_1, \dots,y_n) w_{\{s\}}(m_s; k_lm_l),  
\end{eqnarray}
where $y_l \in (k_lm_l)^M \subseteq k_lM$,  for  $j \neq s$, and $y_s \in k_s^M \subseteq k_sM$. 
Both $w_{\{s\}}(k_sm_s; k_lm_l)$ and $w(y_1, \dots,y_n)$ are trivial by assumption, 
 since they belong to the subgroup $w( k_1M,\dots,k_nM)=1.$ 
By (\ref{11}), we conclude that $ w_{\{s\}}(m_s; k_lm_l)=1$ as well, and the result follows. 
\end{proof}

\begin{corollary}\label{M2} 
Let $w=w(x_1, \dots, x_n)$ be a multilinear commutator and let $I$ be a subset of $\{1, \dots ,n \}$.
Let $M$ be a normal subgroup of  a group $K$. 
 Assume that 
\[ w_J (K; M)=1 \quad \textrm{for every}\ J \subsetneq I.\]
Suppose we are given elements  $k_i \in K$ with $i \in I$ and  elements $m_j \in M$ with $j \in \{1, \dots, n\}$. 
 Then we have 
\[w_I(k_im_i; m_l)=w_I(k_i;m_l).\] 
\end{corollary}
\begin{proof} 
Fix an index $s \in I$ and let $J=I\setminus \{s\}$. Then we can write 
\[w_I(k_im_i; m_l)=w_{\{s\}}(k_sm_s;c_l)\]
where $c_l=k_lm_l$ if $l\in J$, and $c_l=m_l$ if $l \notin I$. 
 Moreover we can write $k_sm_s= h k_s$ for some  $h \in M$.   By Lemma \ref{2.1-conjugates}, 
\begin{eqnarray}\label{12}
 w_{\{s\}}(h k_s; c_l)=w(y_1, \dots,y_n) w_{\{s\}}(k_s; c_l),  
\end{eqnarray}
 where $y_s \in h^K \subseteq M$
 and $y_l \in m_l^K \subseteq M$ if $l \notin I$. 
 In particular 
\[w(y_1, \dots,y_n) \in w_J (K; M).\] 
and thus $w(y_1, \dots,y_n)=1$,  since $w_J (K; M)=1$ by assumption. 
From  (\ref{12}) we deduce that
\[ w_{\{s\}}(k_s m_s; c_l)= w_{\{s\}}(h k_s; c_l)= w_{\{s\}}(k_s; c_l).\]
By repeating the argument for every $s\in I$, we get the desired conclusion.
\end{proof}

\section{Proof of the main theorems}

Our main theorems are both consequences of the following result, which will be proved in this section.

\begin{theorem}\label{together} Let $m$ be a positive integer and $w=w(x_1,\dots ,x_n)$ a multilinear commutator word. Suppose that $G$ is a group having a subgroup $K$ such that $W=w(G)\leq K$ and $|K:C_K(a)|\le m$ for every $a\in G_w$. Then $[w(K),W]$ has $(m,n)$-bounded order.\end{theorem}

For the reader's convenience, the most technical part of our argument is isolated in the following proposition.

\begin{proposition}\label{inductive-step} 
Let $I\subseteq\{1,\dots,n\}$ with $|I|=s$. Under the hypotheses of Theorem \ref{together}, assume that there exist a normal subgroup $N$ of $K$ of
finite order $r$ and  
a normal subgroup $M$ of $K$ of 
 finite index $j$ such that
\[[W, w_J (K; M)]\le N \quad \textrm{for every}\ J \subsetneq I.\]
Then there exist a finite normal subgroup $N_I$ of $K$ of $(r,j,m,s)$-bounded order with
$N\le N_I$ and  
a normal subgroup $M_I$ of $K$ of $(j,m,s)$-bounded index with $M_I\le M$ such that
\[ [W,w_I (K; M_I)]\le N_I.\]
\end{proposition}

\begin{proof}
 Consider a set $C$ of coset representatives of $M$ in $K$, and 
  let $\Omega$ be the set of $n$-tuples  $\underline{c}=(c_1, \dots , c_n)$ where $c_r \in C$ if $r\in I$ and $c_r=1$ otherwise. 
   Notice that the order of $\Omega$ is $j^s$. 
 For any  $n$-tuple  $\underline{c}=(c_1, \dots , c_n) \in \Omega$,  by  Lemma \ref{basic-light}, 
 there exist elements $d_i\in  c_i M$, with
 $i=1,\dots,n$ and a normal subgroup $M_{\underline c}$ of $(j,m)$-bounded index in $K$ such that the order of 
$$[W,w(d_1  M_{\underline c},\dots,d_n M_{\underline c})]^K$$
 is $m$-bounded.
Let 
\begin{eqnarray*}
M_I&=&M \cap \bigg( \bigcap_{\underline{c}\in \Omega}M_{\underline c}\bigg),\\
N_I&=& N \, \prod_{\underline{c}\in \Omega}[W,w(d_1  M_{\underline c},\dots,d_n M_{\underline c})]^K.
\end{eqnarray*}
As $|\Omega|=j^s$,
 it follows that 
 $M_I$ has $(j,m,s)$-bounded index in $K$ and $N_I$ has $(r,j,m,s)$-bounded order.

Let $Z/N_I$ be the center of $WN_I/N_I$ in the quotient group $K/N_I$ and let $\bar K=K/Z$. The image of a subgroup
$U$ of $K$ in $\bar G$ will be denoted by $\bar U$, and similarly for the image of an element.

Let us consider an arbitrary element $w_I(k_i,h_l)\in w_I (K; M_I)$, with $k_i \in K$ and $h_l \in M_I$.
Consider the $n$-tuple  $\underline{c}=(c_1, \dots , c_n) \in \Omega$ defined by $k_i\in c_{i}M$ if $i\in I$ and $c_i=1$ otherwise. 
 Let $d_1,\dots,d_n$ the elements as above, corresponding to the  $n$-tuple  $\underline{c}$. 
 Then 
$$[W,w(d_1  M_I,\dots,d_n M_I)]\le N_I,$$
 that is 
$$\overline{w(d_1  M_I,\dots,d_n M_I)}=1,$$
 in the quotient group $\bar K=K/Z$. 
By Corollary \ref{uno2}, we deduce that 
\begin{equation}\label{step}
 \overline{w_I(d_i  M_I;M_I)}=1. 
\end{equation}
Moreover, as $c_{i}M=d_iM$, we have that $k_i=d_iv_i$ for some $v_i\in M$.
It also follows from our assumptions that  $$\overline{w_J (K; M)}=1$$
for every proper subset $J$ of $I$. Thus we can apply Corollary \ref{M2} and we obtain that 
$$w_I (\overline k_i;\overline h_l)=w_I (\overline d_i\overline v_i;\overline h_l)=
w_I (\overline d_i;\overline h_l)=1,$$
where in the last equality we have used (\ref{step}).
Since $w_I(k_i,h_l)$ was an arbitrary element of $ w_I (K; M_I)$, 
 it follows that $$\overline{w_I (K; M_I)}=1,$$
  that is \[ [W,w_I (K; M_I)]\le N_I,\]
as desired.
\end{proof}

Now the proof of Theorem \ref{together}  is an easy induction.

\begin{proof}[Proof of Theorem \ref{together}.]
We will prove that for every $s=0,\dots,n$  there exist a finite normal subgroup $N_s$ of $K$ of $(m,n)$-bounded order and
a normal subgroup $M_s$ of $K$ of $(m,n)$-bounded index such that
\[ [W,w_I (K; M_s)]\le N_s\]
for every subset $I$ of $\{1,\dots,n\}$ with $|I|\le s$.
Once this is done, the theorem will follow taking $s=n$.

Assume that $s=0$. 
We apply Lemma \ref{basic-light} with $M=K$ and $k_i=1$ for every $i=1,\dots,n$. 
Thus  there exist  $a_1,\dots,a_n \in K$ and  a normal subgroup of $m$-bounded index  $M_0$ of $K$, such that the order of 
$$N_0=[W,w( a_1 M_0 ,\dots , a_n M_0)]^K$$
 is $m$-bounded.

Let $Z/N_0$ be the center of $WN_0/N_o$ in the quotient group $K/N_0$ and let $\bar K=K/Z$.
We have that 
$$\overline{w( a_1 M_0 ,\dots , a_n M_0)}=1,$$
 so it follows from Corollary \ref{uno2} that
$$\overline{w( M_0 ,\dots ,  M_0)}=1,$$
 that is, $[W,w(M_0)]\le N_0$. This proves the case $s=0$.

Now assume  $s\ge 1$. Choose
$I\subseteq\{1,\dots,n\}$ with $|I|=s$. By induction, the hypotheses of Proposition \ref{inductive-step} are satisfied with $N=N_{s-1}$
and $M=M_{s-1}$, so there exist a finite normal subgroup $N_I$ of $K$ of $(m,n)$-bounded order with
$N_{s-1}\le N_I$ and  
a normal subgroup $M_I$ of $K$ of $(m,n)$-bounded index with $M_I\le M_{s-1}$ such that
\[ [W,w_I (K; M_I)]\le N_I.\]

Let
 $$M_s=\bigcap_{|I|=s}M_I, \quad N_s=\prod_{|I|=s}N_I,$$
  where the intersection (resp. the product) ranges over all subsets $I$ of $\{1,\dots,n\}$ of size $s$.

As there is a $n$-bounded number of choices for $I$, it follows that $N_s$ (resp. $M_s$) has  $(m,n)$-bounded order (resp. $(m,n)$-bounded index 
in $K$).  Note that $M_s\le M_{s-1}$ and $N_{s-1}\le N_s$.
Therefore
\[ [W,w_I (K; M_s)]\le N_s\] 
 for every $I\subseteq\{1,\dots,n\}$ with $|I|\le s$. 
This completes the induction step and the proof of the theorem.
\end{proof}

 \begin{proof}[Proof of Theorem \ref{aaa}.]
Let $w=w(x_1,\dots,x_n)$ be a multilinear commutator, and let $G$ be a group such that $|x^G|\leq m$ for every $w$-value $x$ in $G$.  
Apply Theorem \ref{together} with $K=G$. It follows that $[w(G),w(G)]$ has $(m,n)$-bounded order, as desired. 
\end{proof}

 \begin{proof}[Proof of Theorem \ref{bbb}.]
Let $w=w(x_1,\dots,x_n)$ be a multilinear commutator, and let $G$ be a group such that $|x^{w(G)}|\leq m$ for every $w$-value $x$ in $G$.  
Apply Theorem \ref{together} with $K=w(G)$. It follows  that $[w(w(G)),w(G)]$ has $(m,n)$-bounded order, as desired. 
\end{proof}

\end{document}